\documentclass[12pt,reqno]{article}
\usepackage{amsmath,amssymb,amsthm}
\def\R{\mathbb R}  \def\C{\mathbb C} \def\N{\mathbb N}
\newtheorem{thm}{Theorem}[section]
\newtheorem{lem}[thm]{Lemma}

\newtheorem{rem}[thm]{Remark}
\numberwithin{equation}{section}

\title{Local and global well-posedness in $L^{2} (\R^{n} )$
for the inhomogeneous nonlinear Schr\"{o}dinger equation}
\author{{\bf JinMyong An, JinMyong Kim$^*$}\\
\footnotesize{Faculty of Mathematics, {\bf Kim Il Sung} University, Pyongyang, DPR Korea}\\
\footnotesize{$^*$ Corresponding Author: JinMyong Kim, Faculty of Mathematics, {\bf Kim Il Sung} University}\\
\footnotesize{Email: jm.kim0211@ryongnamsan.edu.kp}
}
\date{}
\begin{document}
\maketitle
\begin{abstract}
This paper investigates the local and global well-posedness for the
inhomogeneous nonlinear Schr\"{o}dinger (INLS) equation
\[iu_{t} +\Delta u=\lambda \left|x\right|^{-b} \left|u\right|^{\sigma } u,
u(0)=u_{0} \in L^{2} \left(\R^{n} \right),\]
where $\lambda \in \C$, $0<b<\min \left\{2,{\rm \;
}n\right\}$ and $0<\sigma \le \frac{4-2b}{n} $. We prove the
local well-posedness and small data global well-posedness of the INLS equation
in the mass-critical case $\sigma =\frac{4-2b}{n} $, which have remained open
 until now. We also obtain some local well-posedness results
in the mass-subcritical case $\sigma <\frac{4-2b}{n} $. In order to obtain the
results above, we establish the Strichartz estimates in Lorentz spaces and use
the contraction mapping principle based on Strichartz estimates.
\end{abstract}


\section{Introduction}

In this paper, we consider the Cauchy problem for the inhomogeneous nonlinear
Schr\"{o}dinger (INLS) equation
\begin{equation} \label{GrindEQ__1_1_}
\left\{\begin{array}{l} {iu_{t} +\Delta u=\lambda \left|x\right|^{-b}
\left|u\right|^{\sigma } u{\rm ,}}
\\ {u\left(0,\; x\right)=u_{0}
\left(x\right),} \end{array}\right.
\end{equation}
where $u:{\rm \; }\R\times \R^{n} \to \C$, $u_{{\rm 0}} :{\rm \; \; }\R^{n} \to \C,{\rm \; }b,{\rm \; }\sigma >0$ and $\lambda \in \C$. $\lambda <0$ is the focusing case
and $\lambda >0$ is the defocusing case.

The case $b=0$ is the well-known classic nonlinear Schr\"{o}dinger (NLS) equation
which has been widely studied during the last three decades. One the other
hand, in the end of the last century, it was suggested that in some situations the
beam propagation can be modeled by the inhomogeneous nonlinear Schr\"{o}dinger
equation in the following form:
\begin{equation} \label{GrindEQ__1_2_}
iu_{t} +\Delta u+V\left(x\right)\left|u\right|^{\sigma } u=0.
\end{equation}

We refer the reader to [3, 4, 16, 19] for the physical background and
applications of \eqref{GrindEQ__1_2_}. E.q. \eqref{GrindEQ__1_2_} has attracted
a lot of interest in the last two decades. We also refer the reader to [2,
5-15, 18, 20-22, 24] for recent work on E.q. \eqref{GrindEQ__1_2_}. The INLS
equation \eqref{GrindEQ__1_1_} is a particular case of \eqref{GrindEQ__1_2_}
and it has also been studied by many authors in recent years. But several
challenging technical difficulties arise in its study and therefore there are
many unsolved problems in its study. For instance, even the local
well-posedness for \eqref{GrindEQ__1_1_} with initial data in $L^{2} $ in the
mass-critical case $\sigma =\frac{4-2b}{n} $ is still an open problem (cf. [18]). We refer
the reader to [5-12, 14, 15, 18] for recent work on \eqref{GrindEQ__1_1_}.

Before recalling the existence results for \eqref{GrindEQ__1_1_}, let us give
some information for this equation. The INLS equation \eqref{GrindEQ__1_1_} has
the following equivalent form:
\begin{equation} \label{GrindEQ__1_3_}
u\left(t\right)=S\left(t\right)u_{0} -i\lambda\int _{0}^{t}S\left(t-\tau
\right)\left|x\right|^{-b} \left|u\left(\tau \right)\right|^{\sigma } u\left(\tau \right)d\tau  ,
\end{equation}
where $S\left(t\right)=e^{it\Delta } $ is Schr\"{o}dinger semi-group. Recall that
the INLS equation \eqref{GrindEQ__1_1_} has the following conservation laws:
\begin{equation} \label{GrindEQ__1_4_}
M\left(u\left(t\right)\right):=\left\| u\left(t\right)\right\| _{L^{2}
\left(\R^{n} \right)}^{2}=M\left(u_{0} \right),
\end{equation}
\begin{equation} \label{GrindEQ__1_5_}
E\left(u\left(t\right)\right):=\frac{1}{2} \left\|
\nabla u\left(t\right)\right\| _{L^{2} \left(\R^{n} \right)}^{2} +\frac{\lambda
}{2+\sigma } \left\| \left|x\right|^{-b} \left|u\right|^{\sigma +2} \right\|
_{L^{1} \left(\R^{n} \right)}=E\left(u_{0} \right).
\end{equation}

It is well-known that under a standard scaling argument, the critical power in
$H^{s} \left(\R^{n} \right)$ with $s<\frac{n}{2} $ for \eqref{GrindEQ__1_1_} is
given by
\begin{equation} \label{GrindEQ__1_6_}
\sigma _{s} =\frac{4-2b}{n-2s} .
\end{equation}

When $s<\frac{n}{2} $, corresponding to the critical case, $\sigma
<\frac{4-2b}{n-2s} $ is said to be a subcritical power in $H^{s} \left(\R^{n}
\right)$. If $s\ge \frac{n}{2} $, $\sigma <\infty $ is said to be a subcritical
power in $H^{s} \left(\R^{n} \right)$. Especially if $\sigma =\frac{4-2b}{n} $,
the problem is known as $L^{2} \left(\R^{n} \right)-$critical or mass-critical;
if $\sigma =\frac{4-2b}{n-2} $, it is called $H^{1} \left(\R^{n}
\right)-$critical or energy-critical.

We briefly recall some local and global well-posedness results for the INLS
equation \eqref{GrindEQ__1_1_}. Cazenave [5] studied the well-posedness in
$H^{1} \left(\R^{n} \right)$. Using an abstract theory, he proved that it is
appropriate to seek solution of \eqref{GrindEQ__1_1_} satisfying
\[u\in C\left(\left[0,{\rm \; }T\right),\;H^{1} \left(\R^{n}
\right)\right)\bigcap C^{1} \left(\left[0,{\rm \; }T\right),\;H^{-1}
\left(\R^{n} \right)\right)\]
for some $T>0$. He also proved that any local solution of the defocusing INLS
equation \eqref{GrindEQ__1_1_} with $u_{0} \in H^{1} \left(\R^{n} \right)$
extends globally in time. Using the abstract theory developed by Cazenave [5],
Genoud--Stuart [14] also studied the local and global well-posedness in $H^{1}
\left(\R^{n} \right)$ for the focusing INLS equation. They proved that the
focusing INLS equation \eqref{GrindEQ__1_1_} with $0<b<\min \{ 2,{\rm \; }n\} $
is locally well-posed in $H^{1} \left(\R^{n} \right)$ in the subcritical case,
i.e. in the case $0<\sigma <\sigma ^{*} $ where
\begin{equation} \label{GrindEQ__1_7_}
\sigma ^{*} =\left\{\begin{array}{l} {\frac{4-2b}{n-2} ,{\rm \; }n\ge 3,} \\
{\infty ,{\rm \; }n=1,{\rm \; }2.} \end{array}\right.
\end{equation}
They also established the large data global well-posedness in the case
$0<\sigma <\sigma _{0} $ and the small data global well-posedness in the case
$\sigma _{0} \le \sigma <\sigma ^{*} $, where $\sigma _{0} $ is $L^{2}
\left(\R^{n} \right)-$critical power.
Later, Genoud [15] and Farah [10] also studied the global well-posedness in $H^{1} \left(\R^{n} \right)$ for the
focusing INLS equation \eqref{GrindEQ__1_1_} with $0<b<\min \left\{2,{\rm \;
}n\right\}$ in the case $\sigma _{0} \le \sigma <\sigma ^{*} $ by using some
sharp Gagliardo-Nirenberg inequalities. Recently, Guzm\'{a}n [18] used contraction
mapping principle based on the Strichartz estimates for the first time to
establish the local and global well-posedness of the INLS equation
\eqref{GrindEQ__1_1_}. He proved that the INLS equation \eqref{GrindEQ__1_1_}
with $0<b<\min \left\{2,{\rm \; }n\right\}$ is locally well-posed in $L^{2}
\left(\R^{n} \right)$ in the mass-subcritical case, i.e. $0<\sigma
<\frac{4-2b}{n} $ and that the local solution above extends globally in time by
using the mass-conservation law. He also studied the local and global
well-posedness in $H^{s} \left(\R^{n} \right)$ with $0<s\le \min
\left\{1,{\rm \; }\frac{n}{2} \right\}$. He proved that the INLS equation is
locally well-posed in $H^{s} \left(\R^{n} \right)$ if $0<b<\tilde{2}$, $0<
s\le \min \left\{1,{\rm \; }\frac{n}{2} \right\}$ and $0<\sigma <\tilde{\sigma
}$ where
\begin{equation} \label{GrindEQ__1_8_}
\tilde{2}=\left\{\begin{array}{l} {\frac{n}{3} ,{\rm \; }n=1,{\rm \;
}2,{\rm \; }3,} \\ {2,{\rm \;  }n\ge 4,}
\end{array}\right.   \tilde{\sigma }=\left\{\begin{array}{l} {\sigma _{s} ,{\rm
\; }s<\frac{n}{2} ,} \\ {\infty ,{\rm \; }s=\frac{n}{2} .}
\end{array}\right.
\end{equation}
Moreover, he established the small data global well-posedness in $H^{s}
\left(\R^{n} \right)$ in the case $0<s\le \min \left\{1,{\rm \; }\frac{n}{2}
\right\}$, $0<b<\tilde{2}$ and $\sigma _{0} <\sigma <\tilde{\sigma }$. Later, Dihn [7] improved the local well-posedness result in $H^{1} \left(\R^{n}
\right)$ of [18] by extending the range of $b$. See also Theorem 1.2 of [8].

But the local and global well-posedness in $H^{s} \left(\R^{n} \right)$  with $0\le s<\frac{n}{2}$ for the
INLS equation \eqref{GrindEQ__1_1_} in the critical case, i.e. $\sigma
=\frac{4-2b}{n-2s} $ have remained open until now. See [9] or Remark 1.7 of [18]
for example.

In this paper, we mainly establish the local well-posedness and small data
global well-posedness in $L^{2} \left(\R^{n} \right)$ for the mass-critical INLS
equation which have remained open until now. To arrive at this goal, we
first establish the Strichartz estimates for the Schr\"{o}dinger semi-group in
Lorentz spaces. We know that the Strichartz estimates for the Schr\"{o}dinger
semi-group in Soblev spaces or Besov spaces play an important role in the study of
classic nonlinear Schr\"{o}dinger equation. But these estimates don't give the
sufficient tools to study the mass-critical INLS equation. For this reason, we
establish the Strichartz estimates for the Schr\"{o}dinger semi-group in Lorentz
spaces and by using them, we prove the local well-posedness and small data
global well-posedness in $L^{2} \left(\R^{n} \right)$ for the mass-critical INLS
equation (see Theorem 1.1). In addition, we give the alternative proof of the
local well-posedness in $L^{2} \left(\R^{n} \right)$ for the mass-subcritical
INLS equation (see Theorem 1.3).

Before stating our main results, we define the important numbers which are used
throughout the paper:
\begin{equation} \label{GrindEQ__1_9_}
2^{*} :=\left\{\begin{array}{l} {\frac{2n}{n-2} ,{\rm \; }n\ge 3,} \\ {\infty
,{\rm \; }n=1,{\rm \; }2,} \end{array}\right.
\end{equation}
\begin{equation} \label{GrindEQ__1_10_}
\frac{1}{\gamma (\cdot )} =\frac{n}{2} \left(\frac{1}{2} -\frac{1}{\cdot }
\right).
\end{equation}

The main results are the following local well-posedness and small data global
well-posedness in $L^{2} $ for the mass-critical INLS equation.

\begin{thm}\label{thm 1.1.}
Let $n\in \N$, $0<b<\min \left\{2,{\rm \;
}n\right\}$ and $\sigma =\frac{4-2b}{n} $. If $u_{0} \in L^{2} \left(\R^{n}
\right)$, then there exists $T^{*} =T^{*} \left(u_{0} \right)>0$ such that
\eqref{GrindEQ__1_1_} has a unique solution
\begin{equation} \label{GrindEQ__1_11_}
u\in L_{loc}^{\gamma \left(r\right)} \left(\left[0,{\rm \; }T^{*} \right),{\rm \;
}L^{r,2} \left(\R^{n} \right)\right),
\end{equation}
where
\begin{equation} \label{GrindEQ__1_12_}
r=\frac{n\left(\sigma +2\right)}{n-b} .
\end{equation}
If $T^{*} <\infty $, then
\begin{equation} \label{GrindEQ__1_13_}
L^{\gamma \left(r\right)} \left(\left[0,{\rm \; }T^{*} \right),{\rm \;
}L^{r,2} \left(\R^{n} \right)\right)=\infty .
\end{equation}
Moreover, for any $2\le p<2^{*} $, we have
\begin{equation} \label{GrindEQ__1_14_}
u\in L_{loc}^{\gamma \left(p\right)} \left(\left[0,{\rm \; }T^{*} \right),{\rm
\; }L^{p,2} \left(\R^{n} \right)\right),
\end{equation}
If $\left\| u_{0} \right\| _{2} $ is sufficiently small, then the above
solution is a global one, i.e. $T^{*} =\infty $ and
\begin{equation} \label{GrindEQ__1_15_}
\left\| u\right\| _{L^{\gamma \left(p\right)} \left(\left[0,{\rm \; }\infty
\right),{\rm \; }L^{p,2} \left(\R^{n} \right)\right)} \lesssim\left\| u_{0}
\right\| _{2} ,
\end{equation}
for any $2\le p<2^{*} $.
\end{thm}

The similar statements are valid in the negative time direction and
we omit the details.

\begin{rem}\label{rem 1.2.}
\textnormal{In theorem 1.1, $T^{*} $ depends on not only
$\left\| u_{0} \right\| _{2} $ but also the choice of $u_{0} $ in $L^{2} $,
which is easily seen from the scaling $u\left(t,{\rm \; }x\right)\to u_{\lambda
} \left(t,{\rm \; }x\right):=\lambda ^{\frac{n}{2} } u\left(\lambda ^{2} t,{\rm
\; }\lambda x\right)$. Thus we can't use the mass conservation law to extend
the local solution above to global one.}
\end{rem}

We also have the following local well-posedness in $L^{2} $ for the
mass-subcritical INLS equation.

\begin{thm}\label{thm 1.3.}
Let $n\in \N$, $0<b<\min \left\{2,{\rm \;
}n\right\}$ and $0<\sigma <\frac{4-2b}{n} $. If $u_{0} \in L^{2} \left(\R^{n}
\right)$, then there exists $T=T\left(\left\| u_{0} \right\| \right)>0$ such
that \eqref{GrindEQ__1_1_} has a unique solution $u$ satisfying
\begin{equation} \label{GrindEQ__1_16_}
u\in L^{\gamma \left(r\right)} \left(\left[-T,{\rm \; }T\right],{\rm \;
}L^{r,2} \left(\R^{n} \right)\right),
\end{equation}
where $r$ is given in \eqref{GrindEQ__1_12_}.
Moreover, for any $2\le p<2^{*} $, the solution satisfies
\begin{equation} \label{GrindEQ__1_17_}
u\in L^{\gamma \left(p\right)} \left(\left[-T,{\rm \; }T\right],{\rm \;
}L^{p,2} \left(\R^{n} \right)\right).
\end{equation}
\end{thm}

Using Theorem 1.3 and the mass-conservation law, we
immediately have the large data global well-posedness and we omit the details.

This paper is organized as follows. In Section 2, we introduce some basic notation and function
spaces and recall some useful facts concerned with Lorentz spaces. In Section 3,
we establish the Strichartz estimates in Lorentz spaces. In Section 4, we prove
Theorem 1.1 and Theorem 1.3.

\section{Preliminaries}

Let us introduce the notation used throughout the paper. As usual, we use $\C$, $\R$ and $\N$ to stand for the sets of complex, real and natural numbers, respectively. $C>0$ will denote positive universal constant, which can be different at different places. $a\lesssim b$ means $a\le Cb$ for some constant $C>0$.
We denote by $p'$ the dual number of $p\in \left[1,{\rm \; }\infty \right]$,
i.e. $1/p+1/p'=1$. As in [25], for $s\in \R$ and $1<p<\infty $, we denote by
$H_{p}^{s} \left(\R^{n} \right)$ the inhomogeneous Sobolev spaces. We shall
abbreviate $H_{2}^{s} \left(\R^{n} \right)$ as $H^{s} \left(\R^{n} \right)$. For
$0<p,\; q\le \infty $, we denote by $L^{p,q} \left(\R^{n}
\right)$ the Lorentz space. Note that $L^{p,q} \left(\R^{n}
\right)$ is a quasi-Banach space for $0<p,\; q\le \infty $.  When $p,\; q>1,$  $L^{p,q} \left(\R^{n}
\right)$ can be turned into a Banach space via an equivalent norm.  Note also that $L^{p,p} \left(\R^{n} \right)=L^{p}
\left(\R^{n} \right)$ and that the dual of $L^{p,q}(\R^{n})$ is $L^{p',q'}(\R^{n})$, for $1<p,\; q<\infty$. See [17] for details. For
$I\subset \R$ and $\gamma \in \left[1,{\rm \; }\infty \right]$, we will use the space-time mixed space $L^{\gamma } \left(I,{\rm \; }X\left(\R^{n}
\right)\right)$ whose quasi-norm is defined by
\[{\rm \; }\left\| f\right\|_{L^{\gamma } \left(I,{\rm \; }X\left(\R^{n} \right)\right)}
=\left(\int _{I}\left\| f\right\| _{X\left(\R^{n} \right)}^{\gamma } dt
\right)^{\frac{1}{\gamma } } ,\]
with a usual modification when $\gamma =\infty $, where $X\left(\R^{n} \right)$
is a quasi-normed space on $\R^{n} $ such as Lebesgue space or Lorentz space. If
there is no confusion, $\R^{n} $ will be omitted in various function spaces.

We state some useful properties of Lorentz spaces.

\begin{lem}\label{lem 2.1.} $($[17]$)$
$\left|x\right|^{-\frac{n}{p} } $ is in
$L^{p,\infty } \left(\R^{n} \right)$.
\end{lem}

\begin{lem}\label{lem 2.2.} $($[17]$)$
For all $0<p,{\rm \; }r<\infty $, $0<q\le
\infty $ we have
\begin{equation} \label{GrindEQ__2_1_}
\left\| \left|f\right|^{r} \right\| _{L^{p,q} } =\left\| f\right\| _{L^{pr,qr}
}^{r} .
\end{equation}
\end{lem}

\begin{lem}\label{lem 2.3.} $($[17]$)$
Suppose $0<p\le \infty $ and $0<q<r\le \infty $. Then we have
\begin{equation} \label{GrindEQ__2_2_}
\left\| f\right\| _{L^{p,r} } \le C_{p,q,r} \left\| f\right\| _{L^{p,q} } .
\end{equation}
In other words, $L^{p,q} $ is a subspace of $L^{p,r} $.
\end{lem}

We end this section with recalling H\"{o}lder inequality for Lorentz spaces. See [23, 17] for example.

\begin{lem}\label{lem 2.4.} $($H\"{o}lder inequality for Lorentz spaces$)$
Let $0<p,{\rm \; }q,{\rm \; }r\le\infty $, $0<s_{1} ,{\rm \; }s_{2},{\rm \; }s\le \infty $. Assume that
\begin{equation} \label{GrindEQ__2_3_}
\frac{1}{p} +\frac{1}{q} =\frac{1}{r} , \frac{1}{s_{1} } +\frac{1}{s_{2} }
=\frac{1}{s} .
\end{equation}
Then we have
\begin{equation} \label{GrindEQ__2_4_}
\left\| fg\right\| _{L^{r,s} } \le C_{p,q,s_{1} ,s_{2} } \left\| f\right\|
_{L^{p,s_{1} } } \left\| g\right\| _{L^{q,s_{2} } } .
\end{equation}
\end{lem}

\section{Strichartz estimates in Lorentz spaces }

In this section, we establish the Strichartz estimates in Lorentz spaces
$L^{p,2} $. These estimates play an important role in the study of the
mass-critical INLS equation.

First we obtain the time decay estimates in Lorentz spaces.

\begin{lem}\label{lem 3.1.}
Let $n\in \N, S\left(t\right)=e^{it\Delta }
$ and $2\le p<\infty $. Then
\begin{equation} \label{GrindEQ__3_1_}
\left\| S\left(t\right)f\right\| _{L^{p,2} } \lesssim\left|t\right|^{-n\left({1
\mathord{\left/{\vphantom{1 2}}\right.\kern-\nulldelimiterspace} 2} -{1
\mathord{\left/{\vphantom{1 p}}\right.\kern-\nulldelimiterspace} p} \right)}
\left\| f\right\| _{L^{p',2} } .
\end{equation}
\end{lem}

\begin {proof}
For any $2<p<r<\infty $, we can take $\theta \in\left(0,{\rm \;
}1\right)$ such that $\frac{1-\theta }{{\rm 2}} +\frac{\theta }{r} =\frac{1}{p}$.
It is well-known that
\begin{equation} \label{GrindEQ__3_2_}
\left\| S\left(t\right)f\right\| _{L^{p} } \lesssim\left|t\right|^{-n\left({1
\mathord{\left/{\vphantom{1 2}}\right.\kern-\nulldelimiterspace} 2} -{1
\mathord{\left/{\vphantom{1 p}}\right.\kern-\nulldelimiterspace} p} \right)}
\left\| f\right\| _{L^{p'} } ,
\end{equation}
for every $2\le p\le \infty $. Noticing $L^{p,p} =L^{p} $, we have
\begin{equation} \label{GrindEQ__3_3_}
\left\| S\left(t\right)f\right\| _{L^{2,2} } \lesssim\left\| f\right\| _{L^{2,2} }
,
\end{equation}
\begin{equation} \label{GrindEQ__3_4_}
\left\| S\left(t\right)f\right\| _{L^{r,r} } \lesssim\left|t\right|^{-n\left({1
\mathord{\left/{\vphantom{1 2}}\right.\kern-\nulldelimiterspace} 2} -{1
\mathord{\left/{\vphantom{1 r}}\right.\kern-\nulldelimiterspace} r} \right)}
\left\| f\right\| _{L^{r',r'} } .
\end{equation}
Using \eqref{GrindEQ__3_3_}, \eqref{GrindEQ__3_4_} and the real interpolation
between Lorentz spaces (see Theorem 5.3.1 of [1]): $\left(L^{2,2} ,{\rm \;
}L^{r,r} \right)_{\theta ,{\rm \; }2} =L^{p{\rm ,2}} $, we have
\eqref{GrindEQ__3_1_} for $2<p<\infty $. \eqref{GrindEQ__3_1_} is the same as
\eqref{GrindEQ__3_3_} if $p=2$.
\end{proof}

Using the time decay estimates \eqref{GrindEQ__3_1_} and the well-known dual
estimate techniques (see section 3.2 of [25]), we have the following Strichartz
estimates in Lorentz spaces.

\begin{lem}\label{lem 3.2.}
Let $S\left(t\right)=e^{it\Delta } $, $A_{S}
:=\int _{t_{0} }^{t}S\left(t-\tau \right)\cdot d\tau  $ and $2\le p,{\rm \;
}r<2^{*} $. Then we have
\begin{equation} \label{GrindEQ__3_5_}
\left\| S\left(t\right)\phi \right\| _{L^{\gamma \left(p\right)}
\left(\R,L^{p,2} \right)} \lesssim\left\| \phi \right\| _{2},
\end{equation}
\begin{equation} \label{GrindEQ__3_6_}
\left\| A_{S} f\right\| _{L^{\gamma \left(p\right)} \left(I,L^{p,2} \right)}
\lesssim\left\| f\right\| _{L^{\gamma \left(r\right)^{{'} } } \left(I,L^{r',2}
\right)},
\end{equation}
where $I$ is an interval and $t_{0} \in \bar{I}$.
\end{lem}

\begin{proof}
 We use the well-known dual estimate techniques (cf. [25]) and
the proof of this lemma is very similar to that of Strichartz estimates in
Sobolev or Besov spaces. So we only give the sketch of proof of
\eqref{GrindEQ__3_6_}. We divide the proof into four steps. For convenience, we
assume that $I=\left[0,{\rm \; }T\right)$, for some $T\in \left(0,{\rm \;
}\infty \right)$ and that $t_{0}=0$, the proof being the same in the general
case.\\
\textbf{Step 1.} Using time-decay estimate \eqref{GrindEQ__3_1_} and
Hardy-Littlewood-Sobolev inequality, we immediately have
\begin{equation} \label{GrindEQ__3_7_}
\left\| A_{S} f\right\| _{L^{\gamma \left(p\right)} \left(I,L^{p,2} \right)}
\lesssim\left\| f\right\| _{L^{\gamma \left(p\right)^{{'} } } \left(I,L^{p',2}
\right)} ,
\end{equation}
for $2\le p<2^{*} $. By the same argument we also have
\begin{equation} \label{GrindEQ__3_8_}
\left\| \int _{0}^{t}S\left(s-\tau \right)f(\tau )d\tau  \right\| _{L^{\gamma
\left(p\right)} \left(I,L^{p,2} \right)} \lesssim\left\| f\right\| _{L^{\gamma
\left(p\right)^{{'} } } \left(I,L^{p',2} \right)} ,
\end{equation}
for any $2\le p<2^{*} $ and $s\in \left[0,{\rm \; }T\right)$.\\
\textbf{Step 2.}
\begin{eqnarray}\begin{split}\label{GrindEQ__3_9_}
\left\| A_{S} f\right\| _{2} &=\left(\int
_{0}^{t}S\left(t-\tau \right)f\left(\tau \right)d\tau  ,{\rm \; }\int
_{0}^{t}S\left(t-\sigma \right)f\left(\sigma \right)d\sigma  \right)
\\
& =\int _{0}^{t}\int
_{0}^{t}\left(S\left(t-\tau \right)f\left(\tau \right),{\rm \; }S\left(t-\sigma
\right)f\left(\sigma \right)\right)d\sigma d\tau
\\
&=\int _{0}^{t}\int _{0}^{t}\left(f\left(\tau \right),{\rm \;
}S\left(\tau -\sigma \right)f\left(\sigma \right)\right)d\sigma d\tau
\\
&=\int _{0}^{t}\left(f\left(\tau \right),{\rm \;
}\int _{0}^{t}S\left(\tau-\sigma \right)f(\sigma )d\sigma\right)d\tau,
\end{split}\end{eqnarray}
where we used the property $S\left(t\right)^{*} =S\left(-t\right)$. Applying
Lemma 2.4 in space and H\"{o}lder inequality in time, and using \eqref{GrindEQ__3_8_}, we have
\begin{equation} \label{GrindEQ__3_10_}
\left\| A_{S} f\right\| _{2} \lesssim \left\| f\right\| _{L^{\gamma
\left(p\right)^{{'} } } \left(I,L^{p',2} \right)} \left\| \int
_{0}^{t}S\left(\tau -\sigma \right)f\left(\sigma \right)d\sigma  \right\|
_{L^{\gamma \left(p\right)} \left(I,L^{p,2} \right)} \lesssim\left\| f\right\|
_{L^{\gamma \left(p\right)^{{'} } } \left(I,L^{p',2} \right)}^{2} .
\end{equation}
Using the fact $L^{2,2} =L^{2} $ and \eqref{GrindEQ__3_10_}, we have
\begin{equation} \label{GrindEQ__3_11_}
\left\| A_{S} f\right\| _{L^{\infty } \left(I,{\rm \; }L^{2,2} \right)} \lesssim \left\| f\right\| _{L^{\gamma \left(p\right)^{{'} } } \left(I,L^{p',2}
\right)} .
\end{equation}
By the same argument, we also have
\begin{equation} \label{GrindEQ__3_12_}
\left\| \int _{s}^{T}S\left(s-t \right)f(t)dt  \right\| _{L^{\infty }
\left(I,{\rm \; }L^{2,2} \right)} \lesssim\left\| f\right\| _{L^{\gamma
\left(p\right)^{{'} } } \left(I,L^{p',2} \right)} ,
\end{equation}
for any $2\le p<2^{*} $ and $s\in \left[0,{\rm \; }T\right)$.\\
\textbf{Step 3.} We have
\begin{eqnarray}\begin{split} \label{GrindEQ__3_13_}
\int _{0}^{T}\left(A_{S} f,{\rm \; }\varphi
\left(t\right)\right) dt&=\int _{0}^{T}\int _{0}^{t}\left(S\left(t-\tau
\right)f\left(\tau \right),{\rm \; }\varphi \left(t\right)\right)d\tau dt  \\
&
=\int _{0}^{T}\int _{\tau }^{T}\left(f\left(\tau \right),{\rm
\; }S\left(\tau -t\right)\varphi \left(t\right)\right)dt d\tau
\\
&=\int _{0}^{T}\left(f\left(\tau \right),{\rm \; }\int _{\tau
}^{T}S\left(\tau -t\right)\varphi \left(t\right) dt\right) d\tau .
\end{split}
\end{eqnarray}
Thus using H\"{o}lder inequality, \eqref{GrindEQ__3_12_} and
\eqref{GrindEQ__3_13_}, we have

\begin{eqnarray}\begin{split} \label{GrindEQ__3_14_}
\left|\int _{0}^{T}\left(A_{S} f,{\rm \; }\varphi
\left(t\right)\right) dt\right|&\le \left\| f\right\| _{L^{1} \left(I,{\rm \;
}L^{2} \right)} \left\| \int _{\tau }^{T}S\left(\tau -t\right)\varphi
\left(t\right) dt\right\| _{L^{\infty } \left(I,{\rm \; }L^{2} \right)} \\
&\lesssim\left\| f\right\| _{L^{1} \left(I,{\rm \; }L^{2} \right)}
\left\|\varphi\right\| _{L^{\gamma \left(p\right)^{{'} } } \left(I,L^{p',2} \right)}.
\end{split}
\end{eqnarray}
Using \eqref{GrindEQ__3_14_} and the duality (see Section 1.4.3 of [17]), we can obtain
\begin{equation} \label{GrindEQ__3_15_}
\left\| A_{S} f\right\| _{L^{\gamma \left(p\right)} \left(I,L^{p,2} \right)}
\lesssim\left\| f\right\| _{L^{1} \left(I,{\rm \; }L^{2} \right)} .
\end{equation}
\textbf{Step 4.} We prove \eqref{GrindEQ__3_6_} by using \eqref{GrindEQ__3_7_},
\eqref{GrindEQ__3_11_} and \eqref{GrindEQ__3_15_}. We divide study in two cases.

\textbf{Case 1.} $p\in \left[2,{\rm \; }r\right]$. We can take $\theta \in
\left[0,{\rm \; }1\right]$ such that
\begin{equation} \label{GrindEQ__3_16_}
\frac{1}{p} =\frac{1-\theta }{2} +\frac{\theta }{r} ,{\rm \; } \frac{1}{\gamma
\left(p\right)} =\frac{1-\theta }{\infty } +\frac{\theta }{\gamma
\left(r\right)} .
\end{equation}
It follows from \eqref{GrindEQ__3_7_}, \eqref{GrindEQ__3_11_}, Lemma 2.4 and
H\"{o}lder inequality that
\begin{equation} \label{GrindEQ__3_17_}
\left\| A_{S} f\right\| _{L^{\gamma \left(p\right)} \left(I,{\rm \; }L^{p,2}
\right)} \lesssim \left\| A_{S} f\right\| _{L^{\infty } \left(I,{\rm \; }L^{2,2}
\right)}^{1-\theta } \left\| A_{S} f\right\| _{L^{\gamma \left(r\right)}
\left(I,{\rm \; }L^{r,2} \right)}^{\theta } \lesssim\left\| f\right\| _{L^{\gamma
\left(r\right)^{{'} } } \left(I,L^{r',2} \right)} .
\end{equation}

\textbf{Case 2.} $p>r\ge 2$. We can take $\theta \in \left(0,{\rm \; }1\right)$ such
that
\begin{equation} \label{GrindEQ__3_18_}
\frac{1}{r} =\frac{1-\theta }{2} +\frac{\theta }{p},{\rm \; } \frac{1}{\gamma
\left(r\right)^{{'} } } =\frac{1-\theta }{1} +\frac{\theta }{\gamma
\left(p\right)^{{'} } } .
\end{equation}
Using Theorem 5.1.2 of [1] and the complex interpolation (cf. [23]): $\left[L^{2,2},{\rm \; }L^{p',2}\right]_{\theta } =L^{r',2}$, we have the following complex interpolation:
\begin{equation} \label{GrindEQ__3_19_}
\left[L^{1} \left(I,{\rm \; }L^{2} \right),{\rm \; }L^{\gamma
\left(p\right)^{{'} } } \left(I,L^{p',2} \right)\right]_{\theta } =L^{\gamma
\left(r\right)^{{'} } } \left(I,L^{r',2} \right).
\end{equation}
Hence, \eqref{GrindEQ__3_7_} and \eqref{GrindEQ__3_15_} imply that
\begin{equation} \label{GrindEQ__3_20_}
A_{S} :{\rm \; }L^{\gamma \left(p\right)} \left(I,{\rm \; }L^{p,2} \right)\to
L^{\gamma \left(r\right)^{{'} } } \left(I,L^{r',2} \right)
\end{equation}
is a bounded operator. This completes the proof of \eqref{GrindEQ__3_6_}. The proof of \eqref{GrindEQ__3_5_} is parallel to the proof of \eqref{GrindEQ__3_6_} and we omit the details (cf. [25]).
\end{proof}

\section{Proofs of main results}

In this section, we prove Theorem 1.1 and Theorem 1.3 by using the contraction mapping principle based on Strichartz estimates in Lorentz spaces.

\textbf{Proof of Theorem 1.1 and Theorem 1.3}

We can easily see that $2<r<2^{*}$. Let $T>0$ and $A>0$ which will be chosen later. Since $L^{r,2} $ is
normable, we can assume that $L^{r,2} $ is a Banach space. We define the following complete metric space
\begin{equation} \label{GrindEQ__4_1_}
D=\left\{u\in L^{\gamma \left(r\right)} \left(I,\;L^{r,2} \right):{\rm
\; \; }\left\| u\right\| _{L^{\gamma \left(r\right)} \left(I,\;L^{r,2}
\right)} \le A\right\},
\end{equation}
which is equipped with the metric
\begin{equation} \label{GrindEQ__4_2_}
d\left(u,{\rm \; }v\right)={\rm \; }\left\| u-v\right\| _{L^{\gamma
\left(r\right)} \left(I,\;L^{r,2} \right)},
\end{equation}
where $I=[0, T]$, $I=[-T, T]$ or $I=[0,\infty)$.
We consider the mapping
\begin{equation} \label{GrindEQ__4_3_}
T:{\rm \; }u(t)\to S(t)u_{0} -i\lambda \int _{0}^{t}S(t-\tau
)\left|x\right|^{-b} \left|u(\tau )\right|^{\sigma } u(\tau )d\tau  \equiv
u_{L} +u_{NL} ,
\end{equation}
where
\begin{equation} \label{GrindEQ__4_4_}
u_{L} =S(t)u_{0} , u_{NL} =-i\lambda \int _{0}^{t}S(t-\tau )\left|x\right|^{-b}
\left|u(\tau )\right|^{\sigma } u(\tau )d\tau  .
\end{equation}
Lemma 3.2 yields that
\begin{equation} \label{GrindEQ__4_5_}
\left\| u_{L} \right\| _{L^{\gamma \left(r\right)} \left(I,\;L^{r,2}
\right)} \lesssim\left\| u_{0} \right\| _{2},
\end{equation}
\begin{equation} \label{GrindEQ__4_6_}
\left\| u_{NL} \right\| _{L^{\gamma \left(r\right)} \left(I,\;L^{r,2}
\right)} \lesssim\left\| \left|x\right|^{-b} \left|u\right|^{\sigma } u\right\|
_{L^{\gamma \left(r\right)^{{'} } } \left(I,\;L^{r',2} \right)}.
\end{equation}
\begin{equation} \label{GrindEQ__4_7_}
\left\| Tu-Tv \right\| _{L^{\gamma \left(r\right)} \left(I,\;L^{r,2} \right)} \lesssim\left\| \left|x\right|^{-b} \left|u\right|^{\sigma } u-\left|x\right|^{-b} \left|v\right|^{\sigma }v\right\|_{L^{\gamma \left(r\right)^{'} } \left(I,\;L^{r',2} \right)}.
\end{equation}
Since $\frac{1}{r'} =\frac{\sigma +1}{r} +\frac{b}{n}$, it follows from Lemma 2.1-Lemma 2.4 that
\begin{equation} \label{GrindEQ__4_8_}
\left\| \left|x\right|^{-b} \left|u\right|^{\sigma } u\right\| _{L^{r',2} } \lesssim \left\| \left|x\right|^{-b} \right\| _{L^{\frac{n}{b} ,\infty } } \left\|
u\right\| _{L^{r,\infty } }^{\sigma } \left\| u\right\| _{L^{r,2} } \lesssim \left\| u\right\| _{L^{r,2} }^{\sigma +1} .
\end{equation}
Using Remark 2.6 of [18] and Lemma 2.1-Lemma 2.4, we also have
\begin{eqnarray}\begin{split} \label{GrindEQ__4_9_}
\left\| \left|x\right|^{-b} \left|u\right|^{\sigma }
u-\left|x\right|^{-b} \left|v\right|^{\sigma } v\right\| _{L^{r',2} } &\lesssim \left\| \left|x\right|^{-b} \left(\left|u\right|^{\sigma }
+\left|v\right|^{\sigma } \right)\left(u-v\right)\right\| _{L^{r,2} }
\\
&\lesssim\left(\left\| u\right\| _{L^{r,{\rm 2}}
}^{\sigma } +\left\| v\right\| _{L^{r,{\rm 2}} }^{\sigma } \right)\left\|u-v\right\| _{L^{r,2} } .
\end{split}\end{eqnarray}
\textbf{Case 1.} We consider the mass-critical case $\sigma =\frac{4-2b}{n} $, i.e. we prove Theorem 1.1. We take $I=\left[0,{\rm \; }T\right]$.
Noticing that
\begin{equation} \label{GrindEQ__4_10_}
\frac{1}{\gamma \left(r\right)^{{'} } } =\frac{\sigma +1}{\gamma
\left(r\right)} ,
\end{equation}
and using \eqref{GrindEQ__4_6_}, \eqref{GrindEQ__4_8_} and H\"{o}lder inequality, we immediately have
\begin{equation} \label{GrindEQ__4_11_}
\left\| u_{NL} \right\| _{L^{\gamma \left(r\right)} \left(I,\;L^{r,2}
\right)} \le C\left\| u\right\| _{L^{\gamma \left(r\right)} \left(I{\rm ,\;
}L^{r,2} \right)}^{\sigma +1} .
\end{equation}
Using \eqref{GrindEQ__4_7_}, \eqref{GrindEQ__4_9_}, \eqref{GrindEQ__4_10_} and H\"{o}lder inequality, we also have
\begin{equation}\label{GrindEQ__4_12_}
\left\|Tu-Tv\right\|_{L^{\gamma \left(r\right)} \left(I{\rm ,\;
}L^{r,2} \right)}\le C\left(\left\| u\right\| _{L^{\gamma \left(r\right)}
\left(I,\;L^{r,2} \right)}^{\sigma } +\left\| v\right\| _{L^{\gamma
\left(r\right)} \left(I,\;L^{r,2} \right)}^{\sigma } \right)
\left\|u-v\right\| _{L^{\gamma \left(r\right)} \left(I,\;L^{r,2} \right)} .
\end{equation}
By Strichartz estimates \eqref{GrindEQ__4_5_}, we can see that $\left\|
S(t)u_{0} \right\| _{L^{\gamma \left(r\right)} \left(\left[0,T\right]{\rm ,\;
}L^{r,2} \right)} \to 0$, as $T\to 0$. Take $A>0$ such that $CA^{\sigma } \le
\frac{1}{4} $ and $T>0$ such that
\begin{equation} \label{GrindEQ__4_13_}
\left\| S(t)u_{0} \right\| _{L^{\gamma \left(r\right)}
\left(\left[0,T\right],\;L^{r,2} \right)} \le \frac{A}{2} .
\end{equation}
Then we have
\begin{equation} \label{GrindEQ__4_14_}
\left\| Tu\right\| _{L^{\gamma \left(r\right)} \left(I,\;L^{r,2}
\right)} \le \left\| S(t)u_{0} \right\| _{L^{\gamma \left(r\right)} \left(I{\rm
,\; }L^{r,2} \right)} +C\left\| u\right\| _{L^{\gamma \left(r\right)}
\left(I,\;L^{r,2} \right)}^{\sigma +1} \le A,
\end{equation}
\begin{equation} \label{GrindEQ__4_15_}
\left\| Tu-Tv\right\| _{L^{\gamma \left(r\right)} \left(I,\;L^{r}
\right)}\le 2CA^{\sigma } \left\| u-v\right\| _{L^{\gamma
\left(r\right)} \left(I,\;L^{r,2} \right)} \le \frac{1}{2} \left\|
u-v\right\| _{L^{\gamma \left(r\right)} \left(I,\;L^{r,2} \right)} .
\end{equation}
\eqref{GrindEQ__4_14_} and \eqref{GrindEQ__4_15_} imply that $T: (D,{\rm \; }d)\to (D,{\rm \; }d)$ is a contraction mapping. From Banach fixed
point theorem, there exists a unique solution $u$ of \eqref{GrindEQ__1_1_} in
$(D,\;d)$. Furthermore for any $2\le p<2^{*} $, it follows
from Lemma 3.2, \eqref{GrindEQ__4_8_}, \eqref{GrindEQ__4_10_} and H\"{o}lder
inequality that
\begin{equation} \label{GrindEQ__4_16_}
\left\| u\right\| _{L^{\gamma \left(p\right)} \left(I,\;L^{p,2} \right)}
\lesssim \left\| u_{0} \right\| _{2} +\left\| u\right\| _{L^{\gamma
\left(r\right)} \left(I,\;L^{r,2} \right)}^{\sigma +1},
\end{equation}
which implies $u\in L^{\gamma \left(p\right)} \left(I,\;L^{p,2}
\right)$. Noticing that $u\left(T\right)\in L^{2} $, we can extend the solution
above using a standard argument and we omit the details (cf. [25]). This completes the
proof of the local well-posedness of \eqref{GrindEQ__1_1_} in the mass-critical
case. Analogous to the above, we can prove the small data global well-posedness
and we only sketch the proof. If we take $I=\left(0,{\rm \; }\infty \right)$,
we also have \eqref{GrindEQ__4_5_}, \eqref{GrindEQ__4_11_},
\eqref{GrindEQ__4_12_}. We take $A=2C\left\| u_{0} \right\| _{2} $ and $\delta
=2\left(4C\right)^{-\frac{\sigma +1}{\sigma } } $. If $\left\| u_{0} \right\|
_{2} \le \delta $, i.e. $CA^{\sigma } <\frac{1}{4} $, then we have
\begin{equation} \label{GrindEQ__4_17_}
\left\| Tu\right\| _{L^{\gamma \left(r\right)} \left(I,\;L^{r,2}
\right)} \le \frac{A}{2} +CA^{\sigma +1} \le A,
\end{equation}
\begin{equation} \label{GrindEQ__4_18_}
\left\| Tu-Tv\right\| _{L^{\gamma \left(r\right)} \left(I,\;L^{r,2}
\right)} \le 2CA^{\sigma } \left\| u-v\right\| _{L^{\gamma \left(r\right)}
\left(I,\;L^{r,2} \right)} \le \frac{1}{2} \left\| u-v\right\|
_{L^{\gamma \left(r\right)} \left(I,\;L^{r,2} \right)} .
\end{equation}
So $T: (D,{\rm \; }d)\to (D,{\rm \; }d)$ is a contraction mapping and
there exists a unique solution $u$ in $D$.

\noindent Furthermore, for any $2\le p<2^{*} $, it follows from Lemma 3.2 that
\begin{equation} \label{GrindEQ__4_19_}
\left\| u\right\| _{L^{\gamma \left(p\right)} (\left[0,{\rm \; }\infty
\right),\;L^{p,2} )} \lesssim \left\| u_{0} \right\| _{2} +\left\|
u\right\| _{L^{\gamma \left(r\right)} (\left[0,{\rm \; }\infty \right){\rm ,\;
}L^{r,2} )}^{\sigma +1} \le A=2C\left\| u_{0} \right\| _{2}.
\end{equation}
This completes the proof Theorem 1.1.

\noindent \textbf{Case 2.} We consider the mass-subcritical case $0<\sigma
<\frac{4-2b}{n} $, i.e. we prove Theorem 1.3. The proof of Theorem 1.3 is
similar to one of Theorem 1.1 and we only sketch the proof. Put
$I=\left[-T,{\rm \; }T\right]$. Noticing that
\begin{equation} \label{GrindEQ__4_20_}
\frac{1}{\gamma \left(r\right)^{{'} } } =\frac{\sigma +1}{\gamma
\left(r\right)} +\left(1-\frac{n\sigma +2b}{4} \right),
\end{equation}
and using \eqref{GrindEQ__4_6_}-\eqref{GrindEQ__4_9_}
and H\"{o}lder inequality, we have
\begin{equation} \label{GrindEQ__4_21_}
\left\| u_{NL} \right\| _{L^{\gamma \left(r\right)} \left(I,\;L^{r,2}
\right)} \le CT^{1-{\left(n\sigma +2b\right)
\mathord{\left/{\vphantom{\left(n\sigma +2b\right)
4}}\right.\kern-\nulldelimiterspace} 4} } \left\| u\right\| _{L^{\gamma
\left(r\right)} \left(I,\;L^{r,2} \right)}^{\sigma +1} ,
\end{equation}
\begin{eqnarray}\begin{split} \label{GrindEQ__4_22_}
\left\|Tu-Tv\right\|& _{L^{\gamma \left(r\right)} \left(I{\rm ,\;
}L^{r,2} \right)}\\
&\le CT^{1-{\left(n\sigma +2b\right)
\mathord{\left/{\vphantom{\left(n\sigma +2b\right)
4}}\right.\kern-\nulldelimiterspace} 4} } \left(\left\| u\right\| _{L^{\gamma \left(r\right)} \left(I,\;L^{r,2} \right)}^{\sigma }
+\left\| v\right\| _{L^{\gamma
\left(r\right)} \left(I,\;L^{r,2} \right)}^{\sigma}\right)\left\|
u-v\right\| _{L^{\gamma \left(r\right)} \left(I,\;L^{r,2} \right)} .
\end{split}
\end{eqnarray}
Put $A=2C\left\| u_{0} \right\| _{2} $ and take $T>0$ satisfying
$2CT^{1-{\left(n\sigma +2b\right) \mathord{\left/{\vphantom{\left(n\sigma
+2b\right) 4}}\right.\kern-\nulldelimiterspace} 4} } A^{\sigma } \le {1
\mathord{\left/{\vphantom{1 2}}\right.\kern-\nulldelimiterspace} 2} $. Using
the standard argument, we can prove Theorem 1.3 and we omit the details.



\end{document}